\documentclass[12pt,a4paper]{article}
\usepackage[T1]{fontenc}
\usepackage[british]{babel}
\usepackage[utf8]{inputenc}
\usepackage{amsmath,amssymb,amsthm}
\usepackage{hyperref}
\usepackage{mathrsfs}

\newcommand\Bb{\mathbf{B}}
\newcommand\Pp{\mathcal{P}}
\newcommand\II{\mathcal{I}}
\newcommand\Cc{\mathscr{C}}
\newcommand\Dd{\mathscr{D}}
\newcommand\Ss{\mathcal{S}}
\newcommand\Tt{\mathscr{T}}

\newcommand\XX{\mathfrak{X}}

\DeclareMathOperator{\N}{N}
\DeclareMathOperator{\I}{I}

\DeclareMathOperator{\Oo}{O}

\DeclareMathOperator{\Core}{Core}
\DeclareMathOperator{\Z}{Z}

\DeclareMathOperator{\Syl}{Syl}

\DeclareMathOperator{\Id}{Id}

\DeclareMathOperator{\Aa}{A}

\newtheorem{theorem}{Theorem}
\newtheorem{lemma}[theorem]{Lemma}
\newtheorem{athm}{Theorem}

\theoremstyle{definition}

\newtheorem{proposition}[theorem]{Proposition}
\newtheorem{corollary}[theorem]{Corollary}

\newtheorem*{question}{Question}

\begin{document}

\title{\bf Coset complexes of $p$-subgroups in finite groups}

\author{{\small Huilong Gu, Hangyang Meng\thanks{E-mail: hymeng2009@shu.edu.cn.  }, Xiuyun Guo}\\
  {\small Department of Mathematics, Shanghai University, }\\
     {\small and}\\
    {\small Newtouch Center for Mathematics of Shanghai University,}\\{\small Shanghai 200444, P. R. China}}
\date{}
\maketitle
\begin{abstract}
Let $G$ be a finite group and $p$ be a prime. We denote by $\Cc_p(G)$ the poset of all cosets of $p$-subgroups of $G$. We characterize the homotopy type of the geometric realization $|\Delta\Cc_p(G)|$ for $p$-closed groups $G$, which is motivated by K.S.Brown's Question. We will further demonstrate that $\chi(\Cc_{p}(G)) \equiv |G|_{p’}~~(\text{mod}~~p)$ for any finite group $G$ and any prime $p$.
\end{abstract}

\section{Introduction}

All groups considered in this paper are finite. Let $G$ be a group and $p$
be a prime. Brown~\cite{Brown1975} first introduced $\Ss_p(G)$ the poset of all non-trivial $p$-subgroups of $G$ whose order complex $\Delta\Ss_p(G)$ is called Brown’s complex. Such complex has played an important role in finite group theory due to the fundamental work of Quillen~\cite{Quillen1978}. The homotopy type of the geometric realization $|\Delta\Ss_p(G)|$ (simply say the homotopy type of $\Ss_p(G)$) in Euclidean spaces has been thoroughly and extensively researched in Quillen’s classical paper~\cite{Quillen1978}. He also proposed the well-known Quillen’s conjecture on the contractibility of $\Ss_p(G)$, which states that $\Ss_p(G)$ is contractible if and only if
the largest normal $p$-subgroup $\Oo_p(G)$ of $G$ is nontrivial.

Brown~\cite{Brown2000} studied the poset $\Cc(G)$ of all proper cosets of $G$ and raised many interesting questions to advance this area. One of Brown’s questions on the non-contractibility of $\Cc(G)$ for a group $G$ is settled by Shareshian and Woodroofe~\cite{ShareshianWoodroofe2016}. Furthermore, the poset of proper cosets over some fixed subgroup as an analogue of $\Cc(G)$ is also considered in authors’ recent paper~\cite{WangMengGuo2024}. At the last part of~\cite[Question 8.4]{Brown2000}, Brown also remarked that it would be interesting to investigate the $p$-local analogue of $\Cc(G)$,consisting
of all cosets $Px$ with $P$ a $p$-subgroup of $G$. This is the main motivation of
our paper. Now we denote by
$$\Cc_{p}(G)=\left\{ Hx~|~H\in\Ss_p(G)\cup\{1\},x \in G \right\}$$
the set of all right cosets $Hx$ with $p$-subgroups $H$ (including the identity subgroup) of $G$. Our first result presents the formula of Euler characteristic for the poset $\Cc_{p}(G)$. As a consequence, we characterize the contractibility of $\Cc_{p}(G)$. For convenience, if $\Cc$ is a poset, we shortly write $\chi(\Cc)$ for the Euler characteristic $\chi(\Delta\Cc)$ of the simplicial complex $\Delta\Cc$.

\begin{athm}\label{thm-p-group}
Let $G$ be a group and $p$ be a prime. Then
$$
\chi(\Cc_p(G))=
\begin{cases}
1,&~\text{if $G$ is a $p$-group};\\
-\sum_{H\in\Ss_p(G)\cup\{1\}}\mu(H,G)|G:H|,&~\text{otherwise}.\\
\end{cases}
$$
where $\mu$ is the M\"{o}bius function of the poset $\Ss_p(G)\cup\{1,G\}$. Moreover, the following statements are equivalent.
\begin{itemize}
\item[\emph{(1)}]~$G$ is a $p$-group;

\item[\emph{(2)}]~$\Cc_{p}(G)$ is contractible;

\item[\emph{(3)}]~$\Cc_{p}(G)$ is $(\mathbb{Z})$-acyclic;

\item[\emph{(4)}]~$\chi(\Cc_{p}(G))=1$.
\end{itemize}
\end{athm}

Recall that a group $G$ is called $p$-closed if $G$ has a normal Sylow $p$-
subgroup. Our next result characterizes the homotopy type of $\Cc_p(G)$ for
$p$-closed groups $G$. We denote by $|G|_{p'}$ the largest positive integer coprime to $p$ which divides the order of a group $G$.
\begin{athm}\label{thm-p-closed}
Let $G$ be a group and $p$ be a prime.
Then the following statements are equivalent.
\begin{itemize}
\item[\emph{(1)}]~$G$ is $p$-closed;

\item[\emph{(2)}]~$\Cc_{p}(G)$ is homotopic to a discrete space of $|G|_{p'}$ points;

\item[\emph{(3)}]~$\Cc_{p}(G)$ has exactly $|G|_{p'}$ connected components.
\end{itemize}
In particular, for all cases above, $\chi(\Cc_{p}(G))=|G|_{p'}$.
\end{athm}
It will be naturally asked whether we can characterize $p$-closed groups equivalently by using the Euler characteristic of $\Cc_{p}(G)$ as in Theorem~\ref{thm-p-group}.
But a group $G$ satisfying the property that $\chi(\Cc_{p}(G))=|G|_{p'}$ could not imply that $G$ is $p$-closed. For example, let $G=S_3 \times S_3 \times \cdots \times S_3~(m~\text{times})$, where $m$ is an even integer.
One can obtain that
$\chi(\Cc_2(G))=3^{m}=|G|_{2'}$ and $G$ is not $2$-closed . Up to now, we have not found any counterexamples $G$ for odd primes $p$. For the sake of convenience, throughout this paper, we set
$$\chi_p(G):=\frac{\chi(\Cc_{p}(G))}{|G|_{p'}},$$
and we call $\chi_p(G)$ the $p$-local Euler characteristic of $\Cc_p(G)$, which is an integer by Theorem~\ref{thm-p-group}. It will be interesting to study the following question:

\begin{question}
Describe all groups satisfying $\chi_p(G)=1$.
\end{question}

Recall that, for a prime $p$, a group $G$ is called a $p$-TI-group if
$P\cap P^g=1$ or $P=P^g$ holds for each $g \in G$, where $P$ is a Sylow $p$-subgroup of $G$. Obviously $p$-closed groups are $p$-TI-groups. Such well-known class of groups is described by M. Suzuki~\cite{Suzuki1964} ($p=2$) and C. Y. Ho~\cite{Ho1979} $(p\geq 3)$. We shall give a description for a class of groups $G$ close to $p$-TI-groups, which satisfy $\chi_p(G)=1$ for some prime $p$. 

\begin{athm}\label{thm-p-TI}
Let $G$ be a group such that $\overline{G}=G/\Oo_p(G)$ is isomorphic to the direct product of some $p$-TI-groups for some prime $p$.
Then $\chi_p(G)=1$ if and only if one of the following statement holds:
\begin{itemize}
\item[\emph{(1)}] $G$ is $p$-closed; or
\item[\emph{(2)}] $p=2$ and 
$\overline{G}=A \times X_1 \times \cdots \times X_m$,  where $A$ is a $2'$-subgroup, $X_i \cong S_3$ and $m$ is a positive even integer.
\end{itemize}
\end{athm}

Our last result is an analogue of "Homological Sylow's Theorem", which, due to Brown \cite{Brown1975}, is stated that
$$\chi(\Ss_p(G)) \equiv 1~(\text{mod}~|G|_p)$$
holds for any group $G$ and prime $p$.

\begin{athm}\label{thm-mod-p}
Let $G$ be a non-$p$-closed group and $p$ be a prime.  Then
$$\chi_p(G) \equiv 1~~(\text{mod}~~p^d),$$
where $p^d=\min\{|P:P \cap Q| \mid P,Q \in \Syl_p(G)~\text{with}~P \neq Q\}$.
\end{athm}

\begin{corollary}
Let $G$ be a group and $p$ be a prime.  Then
$$\chi_p(G) \equiv 1~~(\text{mod}~~p).$$
\end{corollary}

\section{Homotopy of $\Cc_p(G)$}

In this section, we review some of the standard facts of order complexes of posets. See the first two parts of Smith's book\cite{Smith2011} for more details.

Let $(\Pp,\leq)$ be a poset (short for partially ordered set). The order complex $\Delta \Pp$ of $\Pp$ consists of all
finite chains in $\Pp$, including the empty set. The chain of length $i$ in $\Pp$ consisting of $i+1$ elements is called an $i$-simplex in $\Delta \Pp$. In particular, the empty set $\varnothing$ is the $(-1)$-simplex. It is well-known that such (abstract) simplicial complex $\Delta \Pp$ can be embedded into a Euclidean topological space as a subspace denoted by $|\Delta \Pp|$, so-called the geometric realization of $\Delta \Pp$. We abbreviate $|\Delta \Pp|$ as $|\Pp|$ here.

A poset map $f : (\mathcal{P},\leq_\mathcal{P})\rightarrow(\mathcal{Q},\leq_\mathcal{Q})$ of posets~$\Pp, \mathcal{Q}$, is a set-theoretical map such that  $x \leq_\mathcal{P} y$ implies $f(x) \leq_\mathcal{Q} f(y).$ It is routine to check that a poset map $f:\mathcal{P}\rightarrow\mathcal{Q}$ naturally induces a simplicial map $\Delta f$ between order complexes $\Delta \Pp$ and $\Delta \mathcal{Q}$, and $f$ also induces a continuous map $|f| : |\mathcal{P}|\rightarrow|\mathcal{Q}|.$  The following lemma gives a sufficient condition for the homotopy between two continuous maps induced by poset maps.

\begin{lemma}\emph{\cite[Lemma~3.1.7]{Smith2011}}\label{lem-homotopic}
Suppose that $f, g: P \to Q$ satisfy $f,g$ are poset maps with $f \leq g$, that is, $f(x)\leq g(x)$ for each $x \in P$. Then $f$ and $g$ are homotopic, that is, $|f|, |g|: |\mathcal{P}| \to |\mathcal{Q}|$ are homotopic.
\end{lemma}

Recall that topological spaces $X$ and $Y$ are homotopy equivalent if there are continuous maps $f: X \to Y$ and $g: Y \to X$ whose compositions satisfy the homotopies $g \circ f \simeq \text{Id}_X$ and $f \circ g \simeq \text{Id}_Y$. For posets $\Pp, \mathcal{Q}$, $\Pp$ and $\mathcal{Q}$ are said to be homotopy-equivalent if $|\Pp|$ and $|\mathcal{Q}|$ are homotopy-equivalent, simply denoted by $\Pp\simeq \mathcal{Q}$.

For a prime $p$ and a group $G$, we denote
$$\II_p(G)=\left\{ P_1\cap P_2 \cdots \cap P_s  \mid P_i \in \Syl_p(G)~\text{for all}~1\leq i\leq s,~\text{and}~s\geq 1\right\}$$
by the set of all intersections of some Sylow $p$-subgroups of $G$.
\begin{lemma}\label{lem-homo-equ}
Let $G$ be a group and $p$ be a prime. Write
$$\Cc(\II_p(G))=\{Hx \mid H \in \II_p(G), x \in G\}.$$
Then $\Cc_p(G)$ is homotopy-equivalent to $\Cc(\II_p(G))$.
\end{lemma}
\begin{proof}
For each $Hx \in \Cc_p(G)$, we define a map $f$ from $\Cc_p(G)$ to $\Cc(\II_p(G))$ by
$$f(Hx):=\bigcap_{H\subseteq P \in \Syl_p(G)}Px=\left(\bigcap_{H\subseteq P \in \Syl_p(G)}P\right)x.$$
Clearly $f(Hx) \in \Cc(\II_p(G))$, so it is well-defined. It is easy to see that $f$ is a poset map, moreover, $f(Hx)\geq Hx$ for each $Hx \in \Cc_p(G)$ and $f(Hx)=Hx$ for each $Hx \in \Cc(\II_p(G))$.
Let $i$ be the inclusion map from $\Cc(\II_p(G))$ to $\Cc_p(G)$. We can obtain that $f\circ i=\Id_{\Cc(\II_p(G))}$ and $i\circ f\geq \Id_{\Cc_p(G)}$, which implies that $\Cc_p(G)$ is homotopy-equivalent to $\Cc(\II_p(G))$.
\end{proof}

\begin{lemma}\label{lem-mod-normal-p-subgroup}
Let $G$ be a group and $p$ be a prime. Suppose that $N$ is a normal $p$-subgroup of $G$. Then $\Cc_p(G)$ is homotopy-equivalent to  $\Cc_p(G/N).$
\end{lemma}
\begin{proof}
For any subset $X\subseteq G$, we write $\overline{X}=XN/N=\{xN \mid x \in X\}$ for the image of $X$ in the quotient group $G/N$.  Since $N$ is a normal $p$-subgroup of $G$, $N$ is contained in the intersection of any Sylow $p$-subgroups of $G$. Hence $\II_p(G/N)=\{X/N \mid X \in \II_p(G)\}$.
Notice that the map $f: \Cc(\II_p(G))\rightarrow \Cc(\II_p(G/N))$ with $f(C)=\overline{C}$ for $C \in \Cc_p(\II_p(G))$ is a bijective poset map. Hence  $\Cc(\II_p(G))\simeq \Cc(\II_p(G/N))$.
It follows from Lemma~\ref{lem-homo-equ} that $\Cc_p(G)\simeq \Cc(\II_p(G))$ and  $\Cc_p(G/N)\simeq \Cc(\II_p(G/N))$. Hence $\Cc_p(G) \simeq \Cc_p(G/N)$, as desired.
\end{proof}

Let $\Cc$ be a finite poset and denote by
$$\I(\Cc)=\{(x,y)\in \Cc \times\Cc \mid x \leq y\}.$$
the subset of $\Cc \times\Cc $ consisting of all pairs $x,y$ in $\Cc$ with $x \leq y$.
Recall that the M\"{o}bius function $\mu$ of $\Cc$ is a function from $\I(\Cc)$ to $\mathbb{Z}$ such that for each pair $(x,y) \in \I(\Cc)$,
$$\sum_{x\leq z\leq y}\mu(x,z)=\delta(x,y)=\sum_{x\leq z\leq y}\mu(z,y),$$
where $\delta(x,y)=1$ if $x=y$; and $\delta(x,y)=0$ if $x<y$.
As a consequence of the definition of the M\"{o}bius function, we have
\begin{lemma}\label{lem-mobius-subposet}
Let $\Cc,\Cc'$ be two finite posets with $\Cc\subseteq \Cc'$, and let $\mu,\mu'$ be the M\"{o}bius functions of $\Cc,\Cc'$ respectively.
For $x<y \in \Cc$, suppose that
$$\{z \in \Cc \mid x<z<y\}=\{z \in \Cc' \mid x<z<y\}$$
Then $\mu(x,y)=\mu'(x,y)$.
\end{lemma}

We also recall the Euler characteristic of a (simplicial) complex $K$  here. For a complex $K$, the Euler characteristic of $K$ is defined by
$$\chi(K)=\sum_{i=0}^{d}(-1)^i\#(K_i),$$
where $d$ is the dimension of $K$, $K_i$ denotes the set of all $i$-simplices and $\#(K_i)$ denotes the cardinality of $K_i$. Moreover,
$\widetilde{\chi}(K)=\chi(K)-1$ is called the reduced Euler characteristic of $K$. The following well-known result~\cite[Corollary 2.5.2.]{Barmak2011}, due to P. Hall, is a key connection between the M\"{o}bius function of a poset $\Cc$ and the Euler characteristic of its order complex $\Delta\Cc$. Write $\widehat{\Cc}=\Cc \cup \{0,1\}$ by adding the minimal element $0$ and the maximal element $1$ in $\Cc$. Then
$$\widetilde{\chi}(\Cc)=\widehat{\mu}(0,1),$$
where $\widehat{\mu}$ is the M\"{o}bius function of $\widehat{\Cc}$. As a consequence,
$$\chi(\Cc)=-\sum_{ x \in \Cc} \widehat{\mu}(x,1)=-\sum_{ x \in \Cc} \widehat{\mu}(0,x).$$
Moreover, we have
\begin{lemma}\label{lem-char-mobius}
Let $\Cc$ be a finite poset.  Then
$$\chi(\Cc)=\sum_{(x,y) \in \I(\Cc)}\mu(x,y).$$ where $\mu$ is the M\"{o}bius function of the poset $\Cc$.
\end{lemma}
\begin{proof}
Write $\widehat{\Cc}=\Cc \cup \{0,1\}$ by adding the minimal element $0$ and the maximal element $1$ in $\Cc$, and let $\widehat{\mu},\mu$ be the  M\"{o}bius function of $\widehat{\Cc}, \Cc$ respectively.
Lemma~\ref{lem-mobius-subposet} implies that $\mu$ is the restriction of $\widehat{\mu}$ on $ \I(\Cc)$.
By Hall's Theorem,
\begin{align*}
\chi(\Cc)&=\widehat{\mu}(0,1)+1=-\sum_{0<x\leq 1} \widehat{\mu}(x,1)+1=-\sum_{0<x<1} \widehat{\mu}(x,1)-1+1\\
&=-\sum_{0<x<1} \widehat{\mu}(x,1)=-\sum_{x \in \Cc} \widehat{\mu}(x,1)=-\sum_{x \in \Cc} \left(-\sum_{x \leq y<1} \widehat{\mu}(x,y)\right)\\
&=\sum_{x \leq y~\text{and}~x,y\in \Cc} \widehat{\mu}(x,y)=\sum_{(x,y)\in \I(\Cc)} \widehat{\mu}(x,y)\\
&=\sum_{(x,y) \in \I(\Cc)} \mu(x,y).
\end{align*}
The last equation holds as $\mu$ is the restriction of $\widehat{\mu}$ on $ \I(\Cc)$.
\end{proof}

\begin{lemma}\label{lem-char-Mob-XX}
Let $G$ be a group and $\XX$ be a set consisting of some proper subgroups of $G$. Let $\Cc(\XX)$ be the set of all right cosets $Hx$ of $G$ with $H \in \XX$ and $x \in G$. Then
$$\chi(\Cc(\XX))=-\sum_{H \in \XX} \mu(H,G)|G:H|,$$
where $\mu$ is the M\"{o}bius function of the poset $\XX \cup \{G\}$.
\end{lemma}
\begin{proof}
Write $\widehat{\Cc}=\Cc(\XX) \cup \{0,G\}$ by adding the minimal element $0$ and the maximal element $G$. By Hall’s theorem,
$$\chi(\Cc(\XX))=-\sum_{Hx \in \Cc(\XX)}\widehat{\mu}(Hx,G),$$
where $\widehat{\mu}$ is the the M\"{o}bius function of the poset $\widehat{\Cc}$. For each $Hx \in \Cc(\XX)$, set
$$\Aa=\{Tx \in \Cc(\XX) \mid H \leq T \leq G\}~
\text{and}~\Bb=\{T \in \Cc \mid H \leq T \leq G\},$$
there is a bijective poset map $f$ from $\Bb$ to $\Aa$ with $f(T)=Tx$. Hence $\Bb$ and $\Aa$ are isomorphic as posets. Hence, for each $Hx \in \Cc(\XX)$, we have
$$\mu(H,G)=\widehat{\mu}(Hx,G),$$
where $\mu$ is the the M\"{o}bius function of the poset $\XX \cup \{G\}$. Note that 
for each $H \in \XX$, there are exactly $|G:H|$ cosets of $H$ in $\Cc(\XX)$. Thus
$$\chi(\Cc(\XX))=-\sum_{Hx \in \Cc(\XX)}\widehat{\mu}(Hx,G)=-\sum_{H \in \XX}\mu(H,G)|G:H|,$$
The result is proved.
\end{proof}

\begin{corollary}\label{cor-char-intersection}
Let $G$ be a group and $p$ be a prime. Suppose that $G$ is not a $p$-group. Then
$$\chi(\Cc_p(G))=-\sum_{H \in \II_p(G)} \mu(H,G)|G:H|,$$
where $\mu$ is the M\"{o}bius function of the poset $\II_p(G) \cup \{G\}$.
\end{corollary}
\begin{proof}
It follows from Lemma~\ref{lem-homo-equ} that $\Cc_p(G)$ is homotopy-equivalent to $\II_p(G)$. Hence
$\chi(\Cc_p(G))=\chi(\Cc(\II_p(G)))$. Since $G$ is not a $p$-group, $\II_p(G)$  consists of some proper subgroups of $G$.
By Lemma~\ref{lem-char-Mob-XX}, 
$$\chi(\Cc_p(G))=\chi(\Cc(\II_p(G)))=-\sum_{H \in \II_p(G)}\mu(H,G)|G:H|,$$
as desired.
\end{proof}

\begin{proof}[\bf Proof of Theorem~\ref{thm-p-group}]
Assume that $G$ is a $p$-group; $\Cc_p(G)$ has the unique maximal element. Then $\Cc_{p}(G)$ is (canonical) contractible and $\chi(\Cc_p(G))=1$. Hence $(2)$ follows from $(1)$. Clearly $(2)$ implies $(3)$, and $(3)$ implies $(4)$ by definition. Now we only have to show that $(4)$ can implies $(1)$. Assume that $G$ is not a $p$-group; Then
$$1=\chi(\Cc_p(G))=-\sum_{H \in \Ss_p(G)\cup \{1\}} \mu(H,G)|G:H|,$$
where $\mu$ is the M\"{o}bius function of the poset $\Ss_p(G)\cup \{1\}$. Note that, for each $H \in  \Ss_p(G)\cup \{1\}$, $\mu(H,G)$ is an integer and $H$ is a proper $p$-subgroup of $G$. Hence $|G|_{p'}$ divides $|G:H|$, which implies that $|G|_{p'}=1$. Thus $G$ is a $p$-group, as desired.
\end{proof}

\begin{proof}[\bf Proof of Theorem~\ref{thm-p-closed}]
Let $P$ be a Sylow $p$-subgroup of $G$ and denote by $\Omega$ the set of all right cosets of $P$ in $G$.

We firstly show that $(1)$ implies $(2)$. Assume that $G$ is $p$-closed.
Note that, in this case, $P \unlhd G$ is the unique Sylow $p$-subgroup of $G$. Hence $\Omega$ is exactly the set of all maximal elements of $\Cc_p(G)$. By Lemma~\ref{lem-homo-equ}, $\Cc_p(G)$ is homotopic to $\Omega$, which is a discrete space of $|G|_{p'}$ points, as desired.

Notice that $(3)$ consequently follows from $(2)$,  and we only show that
$(3)$ can imply $(1)$. For each $Qx \in \Cc_p(G)$, as $G$ is the union of all right cosets of $P$ in $G$, there exists $Py \in \Omega$ such that $Qx \cap Py\neq \varnothing$. Take $z \in Qx \cap Py$, we will see
$$Qx=Qz \geq \{z\}\leq Pz=Py.$$
Hence $Qx, Py$ stay in the common connected component. Hence we may write $\Tt=\{[Pt]: Pt \in \Omega\}$ as the set of all connected components of $\Cc_p(G)$, where $[Pt]$ is the connected component in $\Cc_p(G)$ containing $Pt$.
Clearly $|\Tt|\leq |\Omega|=|G|_{p'}$.

If $P$ is not normal in $G$, then we can take $P^a\neq P$ for some $a \in G$. Let $t \in P^a \setminus P$. Note that $t$ is a $p$-element and so $\langle t\rangle  \in \Cc_p(G)$. It follows that
$$ P\geq \{1\} \leq \langle t\rangle \geq \{t\} \leq Pt.$$
Hence $[P]=[Pt]$. As $Pt \neq P$, $|\Tt|\leq |\Omega|-1=|G|_{p'}-1$, contrary to the hypothesis that $\Cc_p(G)$ has exactly $|G|_{p'}$  connected components. Hence $G$ is $p$-closed and the proof is complete.
\end{proof}

\section{Proof of Theorem~\ref{thm-p-TI}}
In this section, we will study the $p$-local Euler characteristic for finite groups which are the direct product of $p$-TI-groups for the proof of Theorem~\ref{thm-p-TI}.

Recall that the direct product $\Cc\times \Dd$ of two posets $\Cc$ and $\Dd$ is the poset whose underlying set is the cartesian product $\{(x,y)\mid x\in \Cc,y\in \Dd\}$ and whose order relation is given by $(x_1,y_1)\leq_{\Cc\times \Dd}(x_2, y_2)$ if and only if $x_1\leq_{\Cc}x_2$ and $y_1\leq_{\Dd}y_2$.
\begin{lemma}\emph{\cite[Proposition 5]{Rota1964}}\label{lem-mobius-directprod}
Let $\Cc$ and $\Dd$ be two finite posets and let $\left(x_1, y_1\right) \leq\left(x_2, y_2\right)$ be in $\Cc \times \Dd$. Then
$$\mu_{\Cc \times \Dd}\left(\left(x_1, y_1\right),\left(x_2, y_2\right)\right)=\mu_{\Cc}\left(x_1, x_2\right) \mu_{\Dd}\left(y_1, y_2\right),$$
where $\mu_{\Cc}, \mu_{\Dd}~\text{and}~\mu_{\Cc \times \Dd}$ are the M\"{o}bius functions of $\Cc,\Dd$ and $\Cc \times \Dd$, respectively.
\end{lemma}
\begin{lemma}\label{lem-char-directprod}
Let $\Cc$ and $\Dd$ be a finite poset. Then
$$\chi(\Cc\times \Dd)=\chi(\Cc)\chi(\Dd).$$
\end{lemma}
\begin{proof}
Let $\mu_{\Cc},\mu_{\Dd}~\text{and}~\mu$ be the M\"{o}bius functions of $\Cc,\Dd$ and $\Cc \times \Dd$, respectively. By Lemma~\ref{lem-char-mobius} and Lemma~\ref{lem-mobius-directprod}, we can get that
\begin{align*}
\chi(\Cc\times \Dd)
&=\sum_{((x_1,x_2),(y_1,y_2))\in \I(\Cc\times \Dd)}\mu((x_1,x_2),(y_1,y_2))\\
&=\sum_{((x_1,x_2),(y_1,y_2))\in \I(\Cc\times \Dd)}\mu_{\Cc}(x_1,y_1)\mu_{\Dd}(x_2,y_2)\\
&=\sum_{(x_1,y_1)\in \I(\Cc)}\mu_{\Cc}(x_1,y_1)\sum_{(x_2,y_2)\in \I(\Dd)}\mu_{\Dd}(x_2,y_2)\\
&=\chi(\Cc)\chi(\Dd)
\end{align*}
The proof is complete.
\end{proof}
\begin{lemma}
Let $G_1$ and $G_2$ be two finite groups and $p$ be a prime. Then $\Cc_p(G_1)\times \Cc_p(G_2)$ is homotopy-equivalent to
 $\Cc_p(G_1\times G_2)$. In particular,
 $$\chi(\Cc_p(G_1\times G_2))=\chi(\Cc_p(G_1))\chi(\Cc_p(G_2)).$$
\end{lemma}
\begin{proof}
Let $p_i: G_1\times G_2\rightarrow G_i$ be a projection such that $p_i(x_1,x_2)=x_i$, where $i=1,2$.
If $H$ is a $p$-subgroup of $G_1 \times G_2$, then $p_1(H)$ and $p_2(H)$ are $p$-subgroups of $G_1$ and $G_2$ respectively, and clearly $H \leq p_1(H) \times p_2(H)$.
Write $\mathcal{A}=\Cc_p(G_1)\times \Cc_p(G_2)$ and $\mathcal{B}=\Cc_p(G_1\times G_2)$. Set maps
$$\psi:\mathcal{A}\rightarrow \mathcal{B}; (Xx,Yy)\mapsto (X \times Y)(x,y);$$
$$\phi:\mathcal{B}\rightarrow \mathcal{A}; H(x,y)\mapsto (p_1(H)x, p_2(H)y);$$
where $X,Y,H$ are $p$-subgroups of $G_1,G_2, G_1 \times G_2$ respectively. It is not difficult to check that $\phi\circ\psi=\Id_{\mathcal{A}}$ and for each $H(x,y) \in \Cc_p(G_1 \times G_2)$,
$$\psi \circ \phi (H(x,y))=(p_1(H) \times p_2(H))(x,y)\geq H(x,y),$$
which implies that $\psi \circ \phi \geq \Id_{\mathcal{B}}$.  By Lemma~\ref{lem-homotopic},  $\psi \circ \phi \simeq\Id_{\mathcal{B}}$. Thus $\Cc_p(G_1)\times \Cc_p(G_2)$ is homotopy-equivalent to $\Cc_p(G_1\times G_2)$, moreover,
$$\chi(\Cc_p(G_1\times G_2))=\chi(\Cc_p(G_1)\times\Cc_p(G_2)).$$
It follows from Lemma~\ref{lem-char-directprod} that
$\chi(\Cc_p(G_1\times G_2))=\chi(\Cc_p(G_1))\chi(\Cc_p(G_2))$, as desired.
\end{proof}

\begin{proposition}\label{prop-p-TI}
Let $G$ be a $p$-TI-group for some prime $p$.  Then $$\chi(\Cc_p(G))=s|G|_{p'}-(s-1)|G|,$$
where $s$ is the number of all Sylow $p$-subgroups of $G$.
\end{proposition}
\begin{proof}
For a $p$-TI-group $G$, by definition, $\II_p(G)=\Syl_p(G) \cup \{1\}$. Notice that $\mu(P,G)=-1$ for $P \in \Syl_p(G)$, and $\mu(1,G)=s-1$, where $\mu$ is the M\"{o}bius functions of $\II_p(G) \cup \{G\}$. It follows from
Corollary~\ref{cor-char-intersection} that

$$
\begin{aligned}
\chi(\Cc_p(G))&=-\sum_{H \in \II_p(G)} \mu(H,G)|G:H|\\
&=-((-1)s|G|_{p'}+(s-1)|G|)\\
&=s|G|_{p'}-(s-1)|G|.
\end{aligned}
$$
The result is proved.
\end{proof}

\begin{lemma}\label{lem-2-3}
Let $P$ be a Sylow $2$-subgroup of $G$. Suppose that $|P|=2$ and $|G:\N_G(P)|=3$. Then $G\cong S_3\times A$, where $A$ is a $2'$-group.
\end{lemma}
\begin{proof}
Suppose that the result is false and let $G$ be a counterexample with minimal order. Considering the action of $G$ on the set $\Syl_2(G)$ via conjugation, as $|\Syl_2(G)|=|G:\N_G(P)|=3$, one can obtain that $G/K$ is isomorphic to a subgroup of $S_3$, where
$K=\Core_G(\N_G(P))$ is the largest normal subgroup of $G$ contained in $\N_G(P)$, which is also the kernel of this action.  
If $|K|$ is divisible by $2=|G|_2=|P|$, then $P\leq K$. By Frattini Argument, $G=K\N_G(P)=\N_G(P)$, which is a contradiction.  
Hence $K$ is of odd order.

We shall show that $K\neq 1$; Otherwise $G$ is isomorphic to a subgroup of $S_3$. Since $G$ acts on $\Syl_2(G)$ transitively and $|G|_2=2$, we can deduce that $G\cong S_3$, contrary to the choice of $G$.

Now let $N$ be a minimal normal subgroup of $G$ contained in $K$. 
As $K$ is of odd order, Feit-Thompson theorem implies that $N$ is solvable, hence $N$ is elementary Abelian. In the quotient group $G/N$, $PN/N$ is a Sylow $2$-subgroup of $G/N$ with order $2$, and as $N \leq K \leq \N_G(P)$, it implies that
$$|G/N:\N_{G/N}(PN/N)|=|G/N:\N_G(P)/N|=|G:\N_G(P)|=3.$$
The minimality of $G$ implies that $G/N=T/N\times A/N$, where $T,A\leq G$ satisfy that $T/N\cong S_3$ and $A/N$ is a $2'$-subgroup.
Note that $|G:T|$ is not divisible by $2$, $P\leq T$ and $G=\N_G(P)T$ by Frattini Argument. Hence 
$$|T:\N_T(P)|=|T\N_G(P):\N_G(P)|=|G:\N_G(P)|=3.$$
If $T<G$, applying induction on $T$, we have that $T=S \times R$, where $S,R \leq G, S \cong S_3$ and $R$ is a $2'$-subgroup. Note that $P \leq S$ and $R$ centralizes $P$. Hence $R \leq \N_T(P)$ and it implies that $\N_T(P)=\N_S(P)R=PR=P \times R$.
Since $N \leq \N_G(P) \cap T=\N_T(P)$ and $N$ is a $2'$-group, $N \leq R$.
As $SR \cap A=T \cap A=N$, we can and
$S \cap RA \subseteq SR \cap AR=(SR \cap A)R=NR=R$. Hence $S \cap RA \subseteq S \cap R=1$ and $G=SRA=S \times RA$, where $RA$ is a $2'$-group. This is a contradiction. Hence we have $G=T$, so $G/N \cong S_3$.

Let $X/N$ be the Sylow $3$-subgroup of $G/N$ and clearly $X \unlhd G$, moreover, $[X,P]N=X$ as $[X/N,PN/N]=X/N$. Considering the coprime action of $P$ on $X$ via conjugation, there exists a $P$-invariant Sylow $3$-subgroup $Q$ of $X$ and $X=QN$. Note that $N \leq \N_G(P)$ normalizes $P$ and $N \leq G$, hence $[N,P] \leq N \cap P=1$, i.e. $P$ acts trivially on $N$. 

Now we can claim that there exists an element $y$ of $Q$ with order $3$ such that $y \notin N$. Otherwise, every element of $Q$ with order $3$ is in $N$, which is centralized by $P$. By the coprime action, $P$ acts trivially on $Q$, hence trivially on $X=QN$, which implies that $X=[X,P]N=N$. This contradicts that $|X/N|=3$. Hence $X=\langle y\rangle N$ for some $y \in Q$ with order $3$ with $y \notin N$.

Let $P=\langle x\rangle$ for some element $x$ of order $2$. Write $\overline{G}=G/N$. Considering in the quotient group $\overline{G}$, we have 
$\overline{G}=\overline{X}\overline{P}=\langle \overline{y}, \overline{x}\rangle \cong S_3$ and 
$\overline{y}^{ \overline{x}}= \overline{y}^{-1}$. It implies that
$y^x=y^{-1}t$ for some $t \in N$. Recall that $P=\langle x\rangle$ centralizes $N$, so $x^{y}$ also centralizes $N$.
Then, as $o(x)=2,o(y)=3$,
$$x^{y}x=y^{-1}xyx=y^{-1}y^x=y^{-1}y^{-1}t=yt$$
also centralizes $N$. Notice that $t \in N$ and $N$ is Abelian, we obtain that $y$ centralizes $N$. Since $G=\langle x,y\rangle N$,   we can deduce that $N \leq \Z(G)$.

Recall that $y^x=y^{-1}t$ for some $t \in N$. As $t^{-1} \in N$ is of odd order, there exists an element $s \in N$ such that $s^2=t^{-1}$. Set $z=ys$. Since $s \in N \subseteq \Z(G)$, we will see that
$$z^x=(ys)^x=y^xs=y^{-1}ts=y^{-1}s^{-1}=(sy)^{-1}=(ys)^{-1}=z^{-1}.$$
Note that $t^3=1$ as $o(y^x)=o(y)=3$ and $t \in \Z(G)$, which implies that $s^3=1$. Hence $o(z)=3$ and $\langle z,x\rangle \cong S_3$. Note that
$$G=\langle y,x\rangle N=\langle ys,x\rangle N=\langle z,x\rangle N.$$
Since $N \leq \Z(G)$, $\langle z,x\rangle \cap N \leq \Z(\langle z,x\rangle)=1$ as $\Z(S_3)=1$. Hence $G=\langle z,x\rangle \times N \cong S_3 \times N$, which is the final contradiction.
\end{proof}

Recall that we denote by
$$\chi_p(G)=\frac{\chi(\Cc_{p}(G))}{|G|_{p'}},$$
the $p$-local Euler characteristic of $\Cc_p(G)$. Here we summarize some properties:
\begin{lemma}\label{lem-p-local-char}
Let $G,H$ be two groups and $p$ be a prime. Then
\begin{itemize}
\item[\emph{(1)}] $\chi_p(G \times H)=\chi_p(G)\chi_p(H)$;
\item[\emph{(2)}] If $N$ is a normal $p$-subgroup of $G$, then $\chi_p(G/N)=\chi_p(G)$;
\item[\emph{(3)}] If $G$ is a $p$-TI-group, then $\chi_p(G)=s-(s-1)|G|_p$, where $s$ is the number of all Sylow $p$-subgroups of $G$. In particular, if $G$ is $p$-closed, $\chi_p(G)=1$.
\item[\emph{(4)}] If $G$ is a $p$-TI-group and $G$ is not $p$-closed, then $\chi_p(G) \leq -1$, here the equation holds if and only if $p=2$ and $G\cong S_3 \times A$, where $A$ is a $2'$-group.
\end{itemize}
\end{lemma}
\begin{proof}
Part~$(1)$ follows from Lemma~\ref{lem-mobius-directprod} and Part~$(2)$ is a direct calculation. Part~$(3)$
follows from Proposition~\ref{prop-p-TI}. We shall show Part~$(4)$. Since $G$ is not $p$-closed, denote by $s$ the number of all Sylow $p$-subgroups of $G$, we have $s> 1$ and $|G|_p\geq p$. Sylow Theorem implies that $s\geq 1+p$.
Hence
$$\chi_p(G)=s-(s-1)|G|_p\leq s-p(s-1)\leq 1+p-p^2 \leq -1$$
We also see that if $\chi_p(G)=-1$, then $|G|_p=p=2$ and $s=3$.
It follows from Lemma~\ref{lem-2-3} that $G \cong S_3 \times A$, where $A$ is a $2'$-group.
\end{proof}

\begin{proof}[\emph{Proof of Theorem~\ref{thm-p-TI}}]
Note that if $G$ is $p$-closed, by Lemma~\ref{lem-p-local-char}(3), $\chi_p(G)=1$. Assume that
$G/\Oo_2(G)\cong A \times S_3 \times \cdots \times S_3~(m~\text{times})$, where $A$ is a $2$-closed group and $m$ is a positive even integer.
It follows from Lemma~\ref{lem-p-local-char}(3) that $\chi(S_3)=-1$ and $\chi_p(A)=1$.  By Lemma~\ref{lem-p-local-char}~(1) and (2),
$$\chi_2(G)=\chi_2(G/\Oo_2(G))=\chi_2(A)\chi_2(S_3)^m=(-1)^m=1,$$
as $m$ is a positive even integer. Hence the sufficiency is proven.

Now we assume that $G$ is a non-$p$-closed group and
$$\overline{G}=G/\Oo_p(G)=X_1 \times X_2 \times \cdots X_s,$$ where $X_i$ is a $p$-TI-group for each $i$. Since
$$1=\chi_p(G)=\chi_p(X_1)\cdots \chi_p(X_s)$$
and $\chi_p(X_i)$ is an integer, we have that 
$\chi_p(X_i)=\pm 1$ for each $i$. 

If $\chi_p(X_i)=1$ for each $i$, by Lemma~\ref{lem-p-local-char}(4), $X_i$ is $p$-closed. Hence $G/\Oo_p(G)$ is $p$-closed and so is $G$, as desired. Now, without loss of generality, we may assume that
$\chi_i(X_i)=-1$ for $1 \leq i \leq m$ and 
$\chi_i(X_i)=1$ for $m+1 \leq i \leq s$, where $m$ must be a positive even integer as the product of all $\chi_p(X_i)$ is equal to one.
It follows from Lemma~\ref{lem-p-local-char} that
$p=2$, moreover, $X_i$ is $2$-closed for $m+1 \leq i \leq s$; and 
 $X_i \cong S_3 \times A_i$ for some $2'$-subgroup $A_i$ for $1 \leq i \leq m$.

For $m+1 \leq i \leq s$, as $X_i \unlhd \overline{G}$, $\Oo_2(X_i) \leq \Oo_2(\overline{G})=1$. Hence $X_i$ is a $2'$-group for $m+1 \leq i \leq s$. Now we have
$$\overline{G}=X_1 \times \cdots \times X_m \times A,$$
where $A=A_1 \times \cdots \times A_m \times X_{m+1} \times \cdots \times X_s$ is a $2'$-group. This is the statement $(2)$ and so Theorem~\ref{thm-p-TI} is proved.
\end{proof}

\section{Proof of Theorem~\ref{thm-mod-p}}


The following well-known lemma is a generalization of Sylow's Third Theorem.
\begin{lemma}\emph{\cite[Satz~7.9]{Huppert1967}}\label{lem-Sylow-third}
Let $G$ be a group and $p$ be a prime. Suppose that
$$|P : P \cap Q| \geq p^d$$
for each pair $P,Q \in \Syl_p(G)$ with $P \neq Q$.
Then $|\Syl_p(G)|\equiv 1~(\text{mod}~p^d)$.
\end{lemma}

Now we can prove Theorem~\ref{thm-mod-p}.
\begin{proof}[\textbf{\emph{Proof of Theorem~\ref{thm-mod-p}}}]
Set $p^d=\min\{|P:P \cap Q| \mid P,Q \in \Syl_p(G)~\text{with}~P \neq Q\}$. As $G$ is non-$p$-closed, $p^d> 1$.
Write $\mathcal{A}=\II_p(G)\setminus \Syl_p(G)$. For each $X \in  \mathcal{A}$, $X$ must be the intersection of at least two Sylow $p$-subgroups of $G$. Hence $p^d$ divides $|G|_p/|X|$.  By Corollary~\ref{cor-char-intersection}, we have
$$
\begin{aligned}
\chi_p(G)=\frac{\chi(\Cc_p(G))}{|G|_{p'}}&=-\sum_{H \in \II_p(G)} \mu(H,G)|G|_p/|H|\\
&=-\sum_{H \in \Syl_p(G)} \mu(H,G)|G|_p/|H|-\sum_{H \in \mathcal{A}} \mu(H,G)|G|_p/|H|\\
&\equiv-\sum_{H \in \Syl_p(G)} \mu(H,G)|G|_p/|H|~(\text{mod}~p^d)\\
&\equiv-\sum_{H \in \Syl_p(G)} \mu(H,G)~(\text{mod}~p^d)\\
&\equiv|\Syl_p(G)|~(\text{mod}~p^d),
\end{aligned}
$$
where $\mu$ is the M\"{o}bius function of the poset $\II_p(G) \cup \{G\}$, and $\mu(H,G)=-1$ for each $H \in  \Syl_p(G)$. It follows from Lemma~\ref{lem-Sylow-third} that
$$\chi(\Cc_p(G))/|G|_{p'}\equiv|\Syl_p(G)|~\equiv 1~(\text{mod}~p^d),$$
as desired.
\end{proof}

\textbf{Acknowledgement.} The second author is supported by Natural Science Foundation of Shanghai (24ZR1422800) and National Natural Science Foundation of China (12471018). The third author is supported by National Natural Science Foundation of China (12171302).

\bibliographystyle{plain}
  \bibliography{bibM}
\end{document}